\documentclass[11pt]{amsart}

\usepackage{a4wide}
\usepackage{color}
\usepackage{esint,amssymb}
\usepackage{graphicx}
\usepackage{mathtools}
\usepackage[colorlinks=true, pdfstartview=FitV, linkcolor=blue, citecolor=blue, urlcolor=blue,pagebackref=false]{hyperref}
\usepackage{microtype}

\usepackage{bm}
\usepackage{bbm}
\usepackage{scalerel} 
\usepackage{dsfont}
\usepackage[font={footnotesize}]{caption}
\usepackage{empheq}
\usepackage[T1]{fontenc}

\newtheorem{theo}{Theorem}
\newtheorem{propo}{Proposition}
\newtheorem{prop}{Proposition}[section]
\newtheorem{lem}[prop]{Lemma}

\theoremstyle{plain}

\theoremstyle{definition}

\numberwithin{equation}{section}

\newcommand{\R}{\mathbb{R}}

\newcommand{\I}{\mathcal{I}}

\newcommand{\g}{\mathbf{g}}

\def \be{\begin{equation}}
\def \ee{\end{equation}}

\def \t0{\rightarrow 0} 

\def \supp{\mathrm{supp }} 
\def \div{\mathrm{div} \,} 
\def \1{\mathbf{1}} 

\def\cd{\mathsf{c_{\d}}}

\def\({\left(}
\def\){\right)}

\def \KNbeta{K_{N,\beta}}

\def \PNbeta{\mathbb{P}_{N, \beta}} 

\renewcommand{\subset}{\subseteq}

\renewcommand{\subset}{\subseteq}

\renewcommand{\tilde}{\widetilde}

\renewcommand{\div}{\divg}
\newcommand{\indic}{\mathds{1}}

\renewcommand{\hat}{\widehat}

\def\XXint#1#2#3{{\setbox0=\hbox{$#1{#2#3}{\int}$}
     \vcenter{\hbox{$#2#3$}}\kern-.5\wd0}}

\def \XN{{X}_N}

\def\HN{\mathcal{H}_N}
\def\HNV{\mathcal{H}_N^{V}}

\def \FN{F_N}

\def \KNbeta{K_{N,\beta}}

\def \Fluct{\mathrm{Fluct}}
\def \fluct{\mathrm{fluct}}

\def \muv{\mu_V} 

\def\Esp{\mathbb{E}} 

\def \KNbeta{K_{N,\beta}}
\def \ZNbeta{Z_{N,\beta}}

\def \Vt{V_{t}}

\def\I{\mathcal{I}}

\def \muv{\mu_V}

\def\g{\mathsf{g}}
\def\d{\mathsf{d}}

\def\muv{\mu_V}

\def\indic{\mathbf{1}}

\def\div{\mathrm{div}\,}

\def\PotXN{\mathrm{Pot}_N}
\def\D{\mathsf{D}}

\begin{document}
\title{On the maximum of the potential of a general two-dimensional Coulomb gas}
\author{Luke Peilen}
\address{Wachmann Hall 544, Temple University, 1805 N. Broad St., Philadelphia, PA 19103}
\email{luke.peilen@temple.edu}
\date{\today}
\maketitle

\begin{abstract}
We determine the leading order of the maximum of the random potential associated to a two-dimensional Coulomb gas for general $\beta$ and general confinement potential, extending the result of \cite[Theorem 1]{LLZ23}. In the case $\beta=2$, this corresponds to the (centered) log-characteristic polynomial of either the Ginibre random matrix ensemble for $V(x)=\frac{|x|^2}{2}$ or a more general normal matrix ensemble. The result on the leading order asymptotics for the maximum of the log-characteristic polynomial is new for random normal matrices.

We rely on connections with the classical obstacle problem and the theory of Gaussian Multiplicative Chaos. We make use of a new concentration result for fluctuations of $C^{1,1}$ linear statistics which may be of independent interest.
\end{abstract}

\section{Introduction}

\subsection{The Model}
We are interested in studying the two-dimensional Coulomb gas with general confinement potential. This is an interacting particle system with point masses $\XN=(x_1,x_2,\dots,x_N) \in \R^{2N}$ distributed according to the Gibbs measure
\begin{equation}\label{Gibbs}
d\PNbeta(\XN)=\frac{1}{\ZNbeta}e^{-\beta \HN^V(\XN)}~d\XN,
\end{equation}
where $\HN^V$ is given by 
\begin{equation}\label{Hamil}
\HN^V(\XN)=\frac{1}{2}\sum_{i \ne j}\g (x_i-x_j)+N\sum_{i=1}^N V(x_i).
\end{equation}
$\g$ is the logarithmic kernel
\begin{equation}\label{log}
\g(x)=-\log|x|
\end{equation}
and $V$ assumed to grow sufficiently fast at infinity. $\beta>0$ denotes the inverse temperature, which we take to be order one. 

From Frostman \cite{F35} (see \cite{ST97}), if the potential $V$ is lower semicontinuous, bounded below and satisfies the growth condition
\begin{equation}
\liminf_{|x|\rightarrow +\infty}\frac{V(x)}{\log|x|}>1
\end{equation}
then the continuous approximation of the Hamiltonian 
\begin{equation}\label{Continuous energy}
\I_V(\mu)=\frac{1}{2}\int\int \g(x-y)~d\mu(x)d\mu(y)+\int V(x)~d\mu(x)
\end{equation}
has a unique, compactly supported minimizer $\muv$ (called the \textit{equilibrium measure}) among the set of probability measures on $\R$, characterized by the Euler-Lagrange equation
\begin{equation}\label{Euler-Lagrange equation}
\begin{cases}
\g \ast \muv+V=c_V & \text{on }\Sigma \\
\g \ast \muv+V \geq c_V & \text{otherwise,}
\end{cases}
\end{equation}
where $\Sigma$ denotes the \textit{support} of $\muv$. If $\Sigma$ is connected, we will say we are in the \textit{one-cut} regime; otherwise, $\Sigma$ may consist of multiple connected components, which is usually called the \textit{multi-cut} regime. $c_V$ is a fixed constant depending on the potential $V$. In the following, we denote the corresponding \textit{effective potential} by 
\begin{equation}\label{effective potential}
\zeta_V:=\g \ast \muv+V-c_V
\end{equation}
and will review some important properties in $\S \ref{Assumptions}$. We also denote
\begin{equation}\label{frach}
\mathfrak{h}_0:=\g \ast \muv
\end{equation}
to match the notation of \cite{LLZ23}. Furthermore, if we assume that $V$ is continuous, then the empirical measures $\mu_N$ given by 
\begin{equation}\label{Empirical measures}
\mu_N:=\frac{1}{N}\sum_{i=1}^N \delta_{x_i}
\end{equation}
converge almost surely to $\muv$ under $\PNbeta$ (cf \cite{BZ98}).

With this understanding of the leading order behavior, one can split the Hamiltonian and write (cf \cite[Lemma 2.1]{AS21} or \cite[Lemma 2.2]{LS18})
\begin{align*}
\HNV(\XN)&=N^2\I_V(\muv)+2N\sum_{i=1}^N\zeta_V(x_i)+\FN(\XN,\muv)
\end{align*}
where $\FN(\XN,\muv)$ is a next-order energy defined by 
\begin{equation}\label{next order energy}
\FN(\XN,\muv)=\frac{1}{2}\int\int_{\Delta^c}\g(x-y) \left(\sum_{i=1}^N \delta_{x_i}-N\muv\right)(x) \left(\sum_{i=1}^N \delta_{x_i}-N\muv\right)(y)
\end{equation}
for a configuration of points $\XN \subset \R^N$ and $\Delta \subset \R^2$ denotes the \textit{diagonal} 
\begin{equation}\label{diagonal}
\Delta:=\{(x,y) \in \R^2:x=y\}.
\end{equation}
This allows us to rewrite (\ref{Gibbs}) as 
\begin{equation}\label{noG}
d\PNbeta(\XN)=\frac{1}{\KNbeta}\exp \left(-\beta \left(\FN(\XN,\muv)+2N\sum_{i=1}^N \zeta_V(x_i)\right)\right)~d\XN
\end{equation}
where $\KNbeta$ is the \textit{next-order partition function}
\begin{equation}\label{nopart}
\KNbeta=\int_{\R^{2N}}\exp \left(-\beta \left(\FN(\XN,\muv)+2N\sum_{i=1}^N \zeta_V(x_i)\right)\right)~d\XN.
\end{equation}
To match the discussion in \cite{LLZ23}, we will take this as our definition of the two-dimensional Coulomb gas going forward, and simply refer to $\FN(\XN,\muv)$ as the energy of the system.

\subsection{Main Result}
We are interested in studying the maximum of the \textit{Coulomb gas potential} generated by a point configuration $\XN$ and background measure $\muv$ given by 
\begin{equation}\label{PotXN}
\PotXN(z)=\int \log|z-x|\left(\sum \delta_{x_i}-N\muv\right)(x)
\end{equation}
over a closed disk of radius $r$ centered at $x \in \Sigma$, which we denote by $\D(x,r)$. We have chosen to match the notation of \cite{LLZ23} to emphasize the connection with their result. Note that if the Coulomb gas (\ref{Gibbs}) corresponds to the eigenvalues of a random matrix model, then this potential corresponds to the (centered) log-characteristic polynomial of the matrix model. This has been studied in \cite{Lam20} for the Ginibre random matrix model, and the authors in \cite{LLZ23} studied the two-dimensional Coulomb gas potential in the case $V(x)=\frac{x^2}{2}$; we are interested in extending the law of large numbers that they prove for $\PotXN$ in that regime to general $V$. We accomplish this in the following theorem.

\begin{theo}\label{theo}
Let $V$ satisfy assumptions (\textbf{A1})-(\textbf{A3}), given in $\S \ref{Assumptions}$. Suppose additionally that $\Sigma$ is one-cut and that $V \in C^5(\R^2)$. Let $r>0$ be such that $\D(x,r) \subset \Sigma$ and $\D(x,r) \cap \partial \Sigma=\emptyset$. Then, we have
\begin{equation}\label{thm}
\frac{1}{\log N}\max_{z \in \D(x,r)}\PotXN(z) \rightarrow \frac{1}{\sqrt{\beta}}
\end{equation}
as $N \rightarrow +\infty$, with the convergence in probability.
\end{theo}
The leading order asymptotics above agree with those already known for the Ginibre ensemble (\cite[Theorem 1.1]{Lam20}) and general $\beta$ with quadratic potential (\cite[Theorem 1.1]{LLZ23}). The result is new for the log-characteristic polynomial of normal random matrices, and for the general Coulomb gas. Importantly, in line with the approach used in \cite{LLZ23}, we extend the scope of \textit{nondeterminantal} tools for studying such interacting particle systems.

In order to prove Theorem \ref{theo}, we need to establish the following bound on fluctuations of linear statistics for test functions with a small amount of regularity, which may be of independent interest. The additional assumption $V \in C^5$ is needed for the main theorem, but is not necessary for our fluctuation result. For a test function $\xi$, we let $\Fluct_N(\xi)$ denote the random variable
\begin{equation}\label{fluct}
\Fluct_N(\xi)=\int \xi(x)\left(\sum \delta_{x_i}-N\muv\right)(x).
\end{equation}
$\xi^\Sigma$ denotes the unique bounded harmonic extension of a test function $\xi$ outside of $\Sigma$.
\begin{propo}\label{Coro1}
Let $\xi\in C^{1,1}(\R^2)$ and suppose $\Delta \xi \in C^{0,\alpha}(\Sigma)$ for some $\alpha>0$. Suppose also that there is a constant $C>0$ such that   
\begin{equation*}
|\xi(x)|\leq C(\log|x|+1).
\end{equation*}
Suppose that $\Sigma$ is one-cut. Then
\begin{multline}\label{corcon}
\Esp \left[\exp \left(-\beta tN\Fluct_N(\xi)\right)\right] \lesssim\\
\exp\( \beta N^2t^2\max\(1,\|\Delta \xi\|_{C^{0,\alpha} (\Sigma)}+\|\xi\|_{C^{1,1}(U)}\)^2+\max(1,\beta)N|t|\(\|\Delta \xi\|_{C^{0,\alpha}(\Sigma)}+\|\xi\|_{C^{1,1}(U)}\)\).
\end{multline}

If $\Sigma$ is multi-cut, the same control holds if we assume in addition that 
\begin{equation}\label{MC}
\int_{\partial \Sigma_i}\nabla \xi^\Sigma \cdot \hat{n}=0
\end{equation}
for every connected component $\Sigma_i$ of $\Sigma$.
\end{propo}
We obtain Proposition \ref{Coro1} from the more general Proposition \ref{genfluct}, which we state and prove in Section $\S \ref{MTE}$. As the next subsection makes clear, this is enough regularity to control the fluctuations of $\mathfrak{h}_0$ at order $1$, which is the key technical difficulty in extending \cite[Theorem 1]{LLZ23}.

This is an improvement in required regularity over the CLT results (see in particular \cite[Theorem 1.2]{BBNY19}, \cite[Theorem 1]{LS18}, \cite[Theorem 3]{S23}) and the order one fluctuation bound of \cite[Theorem 1]{S23}, which require at least $C^{2,1}$ regularity on the test function $\xi$. 
It should be noted that some assumption on $\xi$ beyond Lipschitz regularity is necessary; \cite{PG23} shows that concentration of measure happens at order $\sqrt{N}$, and discusses optimality (\cite[Proposition 5.1]{PG23}). It is unclear if one could improve our regularity assumptions with the current methods; \cite[Theorem 1]{RV07} at least suggests that we should be able to obtain order one fluctuations as soon as the test function is $C^1$.

\subsection{Assumptions}\label{Assumptions}
We will need to make some assumptions on $V(x)$ to guarantee a sufficiently regular $\partial \Sigma$ and to guarantee sufficient regularity on $\mathfrak{h}_0$. As discussed before, we first need some regularity and growth on the potential $V$. This is given by the following.
\begin{enumerate}
\item[\textbf{(A1)}] \textbf{-Growth and Regularity: } $V \in C^{2,\alpha}(\R^2)$ for some $\alpha>0$, and 
\begin{equation*}
\liminf_{|x|\rightarrow +\infty}\frac{V(x)}{\log|x|}>1.
\end{equation*}
\end{enumerate}
This guarantees the existence of the equilibrium measure $\muv$ (see \cite{ST97}), and gives us enough regularity to assume that the second derivative of $V$ is uniformly bounded in the \textit{droplet} $\Sigma=\supp(\muv)$. 

Next, it is observed in \cite[Section 2.5]{S15} that the Euler-Lagrange equation (\ref{Euler-Lagrange equation}) for $\mathfrak{h}_0$ can be interpreted as an \textit{obstacle problem} for $\mathfrak{h}_0$, namely
\begin{equation}\label{obst prob}
\min \left((-\Delta)\mathfrak{h}_0, \mathfrak{h}_0-(c_V-V)\right)=0.
\end{equation}
With this interpretation, we can draw on much of the theory of the classical obstacle problem. First, we will need the following \textit{nondegeneracy assumption}.
\begin{enumerate}
\item[\textbf{(A2)}] \textbf{-Nondegeneracy: }
There exists a constant $\lambda>0$ such that
\begin{equation*}
\Delta V\geq \lambda
\end{equation*}
in the \textit{coincidence set} $\{\mathfrak{h}_0-(c_V-V)=0\}$.
\end{enumerate}
This guarantees that $\Sigma=\{\mathfrak{h}_0-(c_V-V)=0\}$, and thus that $\zeta_V$ given by (\ref{effective potential}) is strictly positive in $\Sigma^c$. 

The boundary of the coincidence set, which here is given by $\partial \Sigma$, is known as the \textit{free boundary}. Points on the free boundary are either \textit{regular} or \textit{singular} (see \cite{C98}), and at regular points the regularity of the free boundary is well understood. In particular, it is shown in \cite{C77} that the free boundary is as regular as the obstacle at all \textit{regular points}. Thus, we make the following assumption.

\begin{enumerate}
\item[\textbf{(A3)}] \textbf{-Regularity of the Free Boundary: }
All points of $\partial \Sigma$ are regular in the sense of \cite{C98}. Furthermore, $\partial \Sigma$ is a finite union of $C^{2,\alpha}$ curves.
\end{enumerate}
This is conjectured to be generic in the work of Schaeffer \cite{S74}, and is shown for dimensions $d \leq 4$ (in particular, this is generic for $d=2$) in the recent work \cite[Theorem 1.1]{FRoS20}. It is shown in \cite{C77} that $\partial \Sigma$ is $C^{1,\gamma}$ (see also \cite[Theorem 7]{C98}) in a neighborhood of all regular points; we need the additional regularity of the second derivative to complete the analysis in Proposition \ref{Coro1}.

With assumptions \textbf{(A1)-(A3)}, we have the following regularity on $\mathfrak{h}_0$, which we use repeatedly. The result dates back to \cite{F72}; we use the version stated in \cite{C98} (and proven in \cite{CK80}).
\begin{prop}[\cite{C98}; Theorem 2; \cite{CK80}]\label{potreg}
Suppose $V$ satisfies \textbf{(A1)-(A3)}. Then, there exists a set $U \supset \Sigma$ such that $\mathfrak{h}_0 \in C^{1,1}(U)$.
\end{prop}

\subsection{Connection with Literature}
This note is in direct response to the article \cite{LLZ23}, wherein the authors prove a law of large numbers for the maximum of the potential of a two-dimensional Coulomb gas with quadratic confinement potential. Their result extended the earlier work of \cite{Lam20}, which proved this convergence for the \textit{Ginibre ensemble}, a nonHermitian random matrix ensemble composed of i.i.d. complex standard Gaussian entries. The eigenvalue density for this ensemble corresponds to a $2$d Coulomb gas at inverse temperature $\beta=2$ with confinement potential $V(x)=\frac{|x|^2}{2}$ (cf \cite{F10}). A related question, the asymptotics of the moments of the characteristic polynomial for such matrices, was examined in \cite{WW19}. \cite{LLZ23} thus is an extension of this result to general inverse temperature $\beta>0$, and we in turn extend that result to general potential $V$ which yield a \textit{one-cut} equilibrium measure.

A key ingredient in understanding the potential of the two-dimensional Coulomb gas is a fine understanding of the behavior of fluctuations of linear statistics at small scales. This question was initially studied in \cite{RV07}, and similar questions were asked for generalizations called \textit{normal matrix models} in \cite{AHM11}, \cite{AHM15} and \cite{AKS23}.  This question is still of interest for Coulomb gases with general potential due to various connections with mathematical physics, in particular the fractional quantum Hall effect, cf \cite{STG99}. Approaches to the Coulomb gas using an \textit{electric energy} interpretation have yielded a fruitful understanding of the behavior of Coulomb gases at small scales in two dimensions in \cite{L17}, and in higher dimensions and varying temperature regimes in \cite{AS21}. This led to generalizations of the fluctuation results described above in \cite{LS18} and \cite{S23}. A related study was also accomplished in \cite{BBNY17} and \cite{BBNY19}.

The behavior of the lower bound relies on the theory of Gaussian Multiplicative chaos, as discussed in \cite{CFLW21}. These measures were introduced by Kahane in \cite{K85}, and are a family of random fractal measures associated to a generalized log-correlated Gaussian field $X$ defined formally by 
\begin{equation*}
d\mu^\gamma ``=" \frac{e^{\gamma X(x)}}{\Esp e^{\gamma X(x)}}~dx
\end{equation*}
for parameters $\gamma>0$. Since many random matrix ensembles and, more generally, Coulomb gases behave asymptotically like log-correlated fields, one expects the weak convergence of measures 
\begin{equation}\label{GMC}
d\mu_N^\gamma:=\frac{e^{\gamma \PotXN(x)}}{\Esp e^{\gamma \PotXN(x)}}
\end{equation}
to $\mu^\gamma$ associated to an appropriate limiting Gaussian field. One reason that these results are useful is that the measures $\mu^{\gamma}$ are primarily supported on so-called ``thick points" where the field $X$ is large. Thus, these kinds of convergence results can be used to obtain information about extreme values of $\PotXN$. 

This convergence has been established for the GUE in \cite{CFLW21}, and for the GOE and GSE in \cite{K21}. The proof of the lower bound for $\PotXN$ in \cite[Theorem 1]{LLZ23}, which extends to our case, relies on a convergence as in (\ref{GMC}) for regularizations of $\PotXN$ at scales $\epsilon\downarrow 0$. It is expected that this convergence holds without regularization, although that question is currently open.

Related questions for the one-dimensional log-gas have also seen extensive study. This is an interacting particle system given by (\ref{Hamil}) on the real line, and corresponds to the eigenvalue distributions of certain classical \textit{Hermitian} matrix ensembles (namely the GOE, GUE and GSE in $\beta=1$, $2$ and $4$ respectively). Central limit theorems for fluctuations of linear statistics of the log-gas go back to \cite{J98} and have subsequently been generalized in \cite{BG13}, \cite{BG16}, \cite{Sh13}, \cite{Sh14}, \cite{BLS18}, \cite{BL18}, \cite{Lam21}, \cite{BMP22} and \cite{P24}. Questions regarding the maximum of the potential field have also seen significant study, for instance in the works \cite{ABB17}, \cite{PZ18}, \cite{CMN18} and \cite{PZ22}.


\subsection{Proof Structure and Outline of Paper}
In Section 2, we introduce some important terminology and review key elements of the proof of \cite[Theorem 1]{LLZ23}. We describe how almost all of the necessary steps transfer immediately to our model. As we discuss in Section 2, the entirety of \cite{LLZ23} generalizes immediately once one can show the truncation error estimate
\begin{equation}\label{expg}
\PNbeta\left(\left|\Fluct_N(g)\right|\geq (\log N)^{0.8}\right)\leq \exp \left(-\frac{1}{2}(\log N)^{1.5}\right)
\end{equation}
for a function analogous to the $g$ found in \cite[Proposition 3.2]{LLZ23}. 

As discussed in \cite[Appendix A]{LLZ23}, this follows immediately from showing that both the exponential moments of $\Fluct_N(\mathfrak{h}_0)$ and $\Fluct_N(g-c\mathfrak{h}_0)$ are typically order $1$, where $c$ is chosen to make $\Delta(g-c\mathfrak{h}_0)$ mean zero. Since $\mathfrak{h}_0 \in C^{1,1}(U)$ (Proposition \ref{potreg}) and $\Delta \mathfrak{h}_0=\cd\muv=\Delta V \in C^{0,\alpha}(\Sigma)$, we can apply Proposition \ref{Coro1} to directly find this control for $\mathfrak{h}_0$. Despite the possible lack of regularity of $g$ at points in $\Sigma$, we show that since $\Delta(g-c\mathfrak{h}_0)$ is mean zero and sufficiently regular we can still invert (\ref{IE-MT}) and obtain the requisite fluctuation control. This allows us to establish (\ref{expg}) in our model, leading to Theorem \ref{theo}.

Section 3 is then devoted to the proof of Propositions \ref{genfluct} and \ref{Coro1}, which uses Johansson's method \cite{J98}, the transport approach of \cite{LS18} and a detailed analysis of the resulting expansion with an eye towards minimizing the requisite regularity of our test functions.

\subsection{Acknowledgements}
The author would like to thank Brian Rider and David Padilla-Garza for useful discussions. They would also like to thank Ofer Zeitouni for providing early access to drafts of \cite{LLZ23}. 

\section{Proof of the Main Theorem}

\subsection{Infrastructure from the Quadratic Case}
In this section, we discuss how the proof of \cite[Theorem1]{LLZ23} goes through without issue as soon as the fluctuations of $\mathfrak{h}_0$ are shown to be order $1$ on the level of exponential moments.

\subsubsection{\cite[$\S 2$]{LLZ23}, Preliminaries on Coulomb Gases}
This section quotes important results on local energy laws and fluctuations of linear statistics from \cite{AS21} and \cite{S23}, in addition to providing a useful description and quoting necessary results on comparison of partition functions for perturbations of $\muv$ by transport. All of the results other than \cite[Lemma 2.8]{LLZ23} are already stated for general measure $\muv$ with density in $C^3(\Sigma)$ that satisfies
\begin{equation*}
\frac{3}{4}\leq \muv \leq \frac{3}{2}.
\end{equation*}
The choice of constants are arbitrary, and instead can be rephrased for $0<\lambda \leq \muv \leq \Lambda$, where the constants would then depend on $\lambda$ and $\Lambda$. This assumption is guaranteed by \textbf{(A2)}, since
\begin{equation*}
(\ref{Euler-Lagrange equation}) \implies \muv=\frac{\Delta V}{\cd} \geq \lambda>0
\end{equation*}
and is $C^3(\Sigma)$ on $\Sigma=\supp(\muv)$. 

The one proposition that makes use of specifically the constant and radial nature of the equilibrium measure in the case $V(x)=\frac{|x|^2}{2}$ is \cite[Lemma 2.8]{LLZ23}; a careful reading of their arguments, however show that this is only used to prove the fluctuation control \cite[Corollary A.8]{LLZ23} and yields the critical \cite[Proposition 3.2]{LLZ23}, which we instead here prove by appeal to Proposition \ref{genfluct}.

\subsubsection{\cite[$\S 3$]{LLZ23}, Upper bound for Law of Large Numbers}
This section is devoted to the proof of 
\begin{equation}\label{UB}
\lim_{N \rightarrow \infty}\PNbeta \left(\max_{z \in \mathsf{D}(x,r)}\PotXN(z) \geq \frac{\alpha \log N}{\sqrt{\beta}}\right)=0
\end{equation}
for all $\alpha>1$ and fixed $r$ with $\mathsf{D}(x,r)\subset \Sigma$. The first key step observes that 
\begin{equation*}
\int \Delta \log |x-z|~dx=2\pi,
\end{equation*}
which prevents the authors from using the theorems from \cite[$\S 2$]{LLZ23} which require $\int \Delta f=0$, with $f$ sufficiently smooth. Thus, the authors consider instead
\begin{equation*}
\varphi_{z,\epsilon}=\rho_\epsilon \ast \log|x-z|-g,
\end{equation*}
where $\rho_\epsilon$ denotes the standard mollifier and $g$ is a solution to 
\begin{equation*}
\Delta g=2\pi \chi,
\end{equation*}
with $\int \chi=1$ and $\chi \in C_0^\infty(\D(z,r')$ with $\D(z,r')\subset \D(x,r)$. The proof then uses this function to prove (\ref{UB}) for a regularization of $\PotXN$, the proof of which does not require that $\muv$ be the specific equilibrium measure for $V(x)=\frac{|x|^2}{2}$. The conclusion for $\PotXN$ then follows by comparison, using \cite[Proposition 3.2]{LLZ23}, which follows from \cite[Corollary A.8]{LLZ23} and states that 
\begin{equation}\label{order1}
\log \Esp \left[e^{t \Fluct_N[g]}\right]=O(t+t^2)
\end{equation}
for an implied constant dependent only on $\beta$ and $V$. We will establish (\ref{order1}) by application of Proposition \ref{genfluct} at the end of this section.


\subsubsection{\cite[$\S 3$]{LLZ23}, Lower bound for Law of Large Numbers and Gaussian Multiplicative Chaos}
This section is devoted to the proof of 
\begin{equation}\label{LB}
\lim_{N \rightarrow \infty}\PNbeta \left(\sqrt{\beta}\sup_{z \in \mathcal{U}}\PotXN(z) \geq \alpha \log \frac{1}{\epsilon(N)}\right)=0
\end{equation}
for $\alpha<2$, $\mathcal{U}$ a fixed small ball and $\epsilon(N) \gg N^{-1/2}$, from which the lower bound follows. The method uses the theory of Gaussian multiplicative chaos; (\ref{LB}) is deduced from the convergence of the measures
\begin{equation*}
\mu_k^\gamma:=\frac{e^{\gamma \sqrt{\beta} \Fluct_N[\varphi_{z, e^{-k}}]}}{\Esp \left[e^{\gamma \sqrt{\beta} \Fluct_N[\varphi_{z, e^{-k}}]}\right]}
\end{equation*}
to $\mathrm{GMC}_\gamma$; this result is established by the theory built in \cite[Section 3]{CFLW21} and \cite[Section 2]{LOS18}, coupled with the comparison of partition functions along mesoscopic perturbations with $\int \Delta \phi=0$ drawn on in \cite[$\S 2$]{LLZ23}. In particular, it doesn't require that $\muv$ be the equilibrium measure for $V(x)=\frac{|x|^2}{2}$.

\subsection{Proof of Main Theorem}
\begin{proof}[Proof of Theorem \ref{theo}]
We see that it is sufficient to establish (\ref{order1}). Let $c$ be an order $1$ constant such that 
\begin{equation*}
\int \Delta(g-c\mathfrak{h}_0)=0
\end{equation*}
and let $\xi=g-c\mathfrak{h}_0$, which is also at least Lipschitz on all of $\R^2$ and satisfies $|\xi(x)|\leq C(\log|x|+1)$ (this follows from the compact support of $\chi$ and $\muv$). Consider the equation
\begin{equation*}
\begin{cases}
\div(\psi \muv)=-\frac{1}{\cd}\Delta \xi & \text{in }\Sigma \\
\psi \cdot \hat{n}=0 & \text{on }\partial \Sigma \\
\psi=\psi^\perp & \text{in }U \setminus \Sigma
\end{cases}
\end{equation*}
where $\psi^\perp$ is a $C^{1,\alpha}$ vector field satisfying $\psi \cdot \nabla \zeta_V=0$ in $U \setminus \Sigma$. The interior equation is well-posed since $\int_{\Sigma} \Delta \xi=0$, and $\psi \in C^{1,\alpha}(\Sigma)$ because $\chi$ and $\muv$ are both $\alpha$-H\"older continuous densities in $\Sigma$. The vector field $\psi^\perp$ can be constructed by hand as in the proof of \cite[Lemma 3.4]{LS18}; one considers an extension of the tangential component $\psi$ on $\Sigma$ (which here is just $\psi$) and subtracts off the projection of $\tilde{\psi}$ onto $\nabla \zeta_V$. 

This transport solves (\ref{IE-MT}). Integrating by parts and using $\psi \cdot \nabla \zeta_V=0$ we find 
\begin{multline*}
\psi(x)\cdot \nabla \zeta_V(x)+\xi(x)-\int \nabla \g(x-y)\cdot \psi(y)~d\muv(y) \\
=\xi-\int_{\Sigma} \g(x-y) \(-\frac{1}{\cd}\Delta \xi \)=\xi-\g \ast \left(-\frac{1}{\cd}\Delta \xi\right)
\end{multline*}
since $\Delta \xi$ is only supported in $\Sigma$. However, $\xi-\g \ast \left(-\frac{1}{\cd}\Delta \xi\right)$ grows at most logarithmically at $\infty$ and 
\begin{equation*}
\Delta\(\xi-\g \ast \left(-\frac{1}{\cd}\Delta \xi\right)\)=\Delta \xi-\Delta \xi=0
\end{equation*}
on all of $\R^2$. Thus, by Liouville's theorem this is constant and so 
\begin{equation*} 
\psi(x)\cdot \nabla \zeta_V(x)+\xi(x)-\int \nabla \g(x-y)\cdot \psi(y)~d\muv(y) =c_\xi.
\end{equation*}
Therefore we can apply Proposition \ref{genfluct} to find 
\begin{equation}\label{sub}
\Esp \left[e^{t\Fluct_N[g-c\mathfrak{h}_0]}\right]=e^{O(t+t^2)}.
\end{equation}
Next, $\mathfrak{h}_0 \in C^{1,1}(U)$ and satisfies $\mathfrak{h}_0 \leq C(\log|x|+1)$; granted Proposition \ref{Coro1} then we also have
\begin{equation}\label{pot}
\Esp \left[e^{t\Fluct_N[\mathfrak{h}_0]}\right]=e^{O(t+t^2)}.
\end{equation}
Now, (\ref{sub}) and (\ref{pot}) are comparable so we find 
\begin{equation*}
\Esp \left[e^{t\Fluct_N[g]}\right]=e^{O(t+t^2)}
\end{equation*}
by H\"older, establishing (\ref{order1}) for our case and proving Theorem \ref{theo}.
\end{proof}

\section{Main Technical Estimate}\label{MTE}
The goal of this section is to prove Proposition \ref{Coro1}. We first establish the following. 
\setcounter{propo}{0}
\begin{prop}\label{genfluct}
Let $\xi\in C^{0,1}(\R^2)\cap C^{2,\alpha}(\Sigma)$ for some $\alpha>0$. Suppose also that there is a constant $C>0$ such that  
\begin{equation*}
|\xi(x)|\leq C(\log|x|+1).
\end{equation*}
Suppose further that there exists a Lipschitz vector field $\psi$ and a constant $c_\xi$ such that 
\begin{equation}\label{IE-MT}
\psi(x)\cdot \nabla \zeta_V(x)+\xi(x)-\int \nabla \g(x-y)\cdot \psi(y)~d\muv(y)=c_\xi
\end{equation}
in an open neighborhood $U \supset \Sigma$. Then,
\begin{multline}\label{concentration}
\Esp \left[\exp \left(-\beta tN\Fluct_N(\xi)\right)\right]\lesssim \\
\exp\( \beta N^2t^2\(\( \|\psi\|_{L^\infty}+\|\psi\|_{C^{0,1}}^2\)(1+\|\xi\|_{C^{0,1}})
\)+\max(1,\beta)N|t|\|\psi\|_{C^{0,1}}\),
\end{multline}
where the constant depends only on $V$, $C$ and $U$. In particular, if $|t| \lesssim \frac{1}{N}$ fluctuations are typically order $1$.
\end{prop}
We first truncate $\xi$ so that we only need to examine the fluctuations on the given open neighborhood $U$ of $\Sigma$.
\begin{prop}\label{extfluct}
Let $\xi$ be a measurable function such that 
\begin{equation*}
|\xi(x)|\leq C(\log|x|+1)
\end{equation*}
for some $C>0$. Let $\eta\in C_c^\infty$ be a test function such that 
\begin{equation*}
\begin{cases}
\eta \equiv 1 & \text{in }U  \\
0 \leq \eta \leq 1 & \text{everywhere} \\
\|\eta\|_{C^\infty}\lesssim 1.
\end{cases}
\end{equation*}
Then, there exists a constant $K>0$ such that for all $|t|\leq K\beta N$,
\begin{equation}
\Esp \exp \(t \Fluct[\xi(1-\eta)]\) =\exp(o(t)). 
\end{equation}
\end{prop}
\begin{proof}
We show that the fluctuations of $\xi(1-\eta)$ should be small simply due to the fact that there are not many points outside of $U$. The approach is borrowed from \cite{LS18} and uses the expansion of partition functions in terms of $\zeta_V$. First, observe that 
\begin{align*}
\left|\int \xi(1-\eta) \left(\sum \delta_{x_i}-N\muv\right)\right|&=\left|\sum_{x_i \notin U} \xi(1-\eta)(x_i)\right| \leq \frac{1}{K}\sum_{x_i \notin U}\zeta_V(x_i)
\end{align*}
because $V \gtrsim (1+\epsilon) \log|x|$ (\textbf{A1}) and $\g \ast \muv +\log|x| \rightarrow 0$ as $|x|\rightarrow 0$ implies $\zeta_V \gtrsim \log|x| \gtrsim |\xi(x)|$ as $|x|\rightarrow +\infty$.
Now, 
\begin{equation*}
\Esp \exp \left(t \int \xi(1-\eta) \left(\sum \delta_{x_i}-N\muv\right)\right) \leq \Esp \exp \left(\frac{|t|}{K}\sum_{i=1}^N \zeta(x_i)\right).
\end{equation*}
Choose $t=\pm K\beta N$. Then,
\begin{equation*}
\Esp \exp \left( \pm K\beta N \int \xi(1-\eta)\left(\sum \delta_{x_i}-N\muv\right)\right)\leq \Esp \exp \left(\beta N \sum \zeta(x_i)\right). 
\end{equation*}
Using \cite[(4.12)]{LS15} we find 
\begin{equation*}
\log\Esp \exp \left(\beta N \sum \zeta(x_i)\right)=\log \KNbeta\left(\muv, \frac{1}{2}\zeta\right)-\log \KNbeta(\muv, \zeta)=o(N).
\end{equation*}
H\"older then gives 
\begin{equation*}
\Esp \exp \left(t \int \xi(1-\eta)\left(\sum \delta_{x_i}-N\muv\right)\right)=\exp(o(t))
\end{equation*}
for $|t|\leq K\beta N$.
\end{proof}

We turn to controlling the fluctuations of $\xi$ in $U$. The idea is to use Johansson's method \cite{J98} coupled with a transport approach as in \cite{LS18}. We opt to use the Taylor expansion approach of \cite{BLS18} and \cite{P24} due to the ease with which it allows us to relax the regularity of $\xi$. First, we expand the fluctuations along a transport.
\begin{prop}\label{expansion}
Let $\xi$ be a measureable test function. Then,
\begin{equation}\label{expansion1}
\Esp_{\PNbeta}\left[\exp\(-\beta t N \Fluct_N(\xi)\)\right]=e^{\beta tN^2\int \xi~d\muv}\frac{\ZNbeta^{\Vt}}{\ZNbeta},
\end{equation}
and
\begin{equation*}
\ZNbeta^{\Vt}=\int \exp \left(-\beta \left(\sum_{i \ne j}\frac{1}{2}\g(x_i-x_j)+N\sum V_t(x_i)\right)\right)~d\XN.
\end{equation*} 
Furthermore, we can expand
\begin{equation}\label{splitt}
\Esp_{\PNbeta}\left[\exp\(- \beta t N\Fluct_N(\xi)\)\right]=e^{T_0}\Esp_{\PNbeta}\left[\exp \left(T_1+T_2\right)
\right],
\end{equation}
where
\begin{align}
\nonumber T_0&=-\beta N^{2}\biggl(\frac{1}{2}\iint (\g(\phi_t(x)-\phi_t(y))-\g(x-y))~d\muv(x)d\muv(y)+\int (V_t\circ \phi_t-V)(x)~d\muv(x)\biggr) \\
\label{t0}&+N \int \log \det D\phi_t(x)~d\muv(x)+t\beta N^2 \int \xi(x)~d\muv(x)\\
\label{t1} T_1&= -\beta N\int \bigg( \int (\g(\phi_t(x)-\phi_t(y))-\g(x-y))~d\muv(y)+(V_t \circ \phi_t -V)(x)+\log \det D\phi_t\biggr)~d\fluct_N(x) \\
\label{t2}T_2&=-\frac{\beta}{2}\int \int_{\Delta^c}(\g(\phi_t(x)-\phi_t(y))-\g(x-y))~d\fluct_N(y)d\fluct_N(x).
\end{align}

\end{prop}

\begin{proof}
With the change of variables $y_i=\phi_t(x_i)$ and $\phi_t:\R^d\rightarrow \R^d$ the map $\text{Id}+t\psi$ we obtain 
\begin{align*}
& \Esp_{\PNbeta}\left[\exp\(- \beta t N\Fluct_N(\xi)\)\right]=\frac{\exp\( \beta  t N^{2} \int \xi~d\muv\) }{\ZNbeta}
\\ & \int \exp \biggl(-\beta  \left(\sum_{i \ne j}\frac{1}{2}\g(\phi_t(x_i)-\phi_t(x_j))+N\sum_{i=1}^NV_t(\phi_t(x_i))\right)+\sum_i \log \det D\phi_t(x_i)\biggr)~d\XN \\
&=\exp\(- \beta  t N^{2} \int \xi~d\muv\) \Esp_{\PNbeta} \biggl[\exp \biggl(-\beta \biggl( \frac{1}{2}\sum_{i \ne j}(\g(\phi_t(x_i)-\phi_t(x_j))-\g(x_i-x_j))+ \\
&\hspace{2cm}N \sum_{i=1}^N (V_t\circ \phi_t-V)(x_i)\biggr)+\sum_{i=1}^N \log \det D\phi_t(x_i)\biggr)\biggr].
\end{align*}
Writing  $$\fluct_N:=\sum_{i=1}^N \delta_{ x_i}-N\muv$$ we find
\begin{equation*}
\Esp_{\PNbeta}
\left[\exp\(- \beta t N \Fluct_N(\xi)\)\right]
=e^{T_0}\Esp_{\PNbeta}\left[\exp \left(T_1+T_2\right)\right]
\end{equation*}
with 
\begin{align*}
T_0&=-\beta N^{2}\biggl(\frac{1}{2}\iint (\g(\phi_t(x)-\phi_t(y))-\g(x-y))~d\muv(x)d\muv(y)+ \\
&\int (V_t\circ \phi_t-V)(x)~d\muv(x)\biggr) +N \int \log \det D\phi_t(x)~d\muv(x)+ t  \beta N^{2}  \int \xi(x)~d\muv(x)\\
T_1&= -\beta N\int \left( \int (\g(\phi_t(x)-\phi_t(y))-\g(x-y))~d\muv(y)+(V_t \circ \phi_t -V)(x)\right)~d\fluct_N(x) \\
&+\int \log \det D\phi_t ~d\fluct_N(x)\\
T_2&=-\frac{\beta}{2}\iint_{\Delta^c}(\g(\phi_t(x)-\phi_t(y))-\g(x-y))~d\fluct_N(y)d\fluct_N(x).
\end{align*}
\end{proof}

First, we control $T_0$.
\begin{prop}[Control of $T_0$]\label{Control of t0}
Suppose $\psi$ is Lipschitz and let $T_0$ be as in (\ref{t0}). Then,
\begin{equation}\label{t0 control}
|T_0|\lesssim \beta N^2t^2\(\|\psi\|_{C^{0,1}}^2+\|\xi\|_{C^{0,1}}\|\psi\|_{L^\infty}\)+N|t|\|\psi\|_{C^{0,1}}.
\end{equation}
In particular, if we take $t \sim \frac{1}{N}$, $T_0$ is order $1$.
\end{prop}
\begin{proof}
This follows from the definition of $\phi_t$ and a first order Taylor expansion. We first rewrite
\begin{multline}
\frac{1}{2}\iint -\log|x+t\psi(x)-y-t\psi(y)|~d\muv(x)d\muv(y)-\frac{1}{2}\iint-\log|x-y|~d\muv(x)d\muv(y)
= \\
\frac{1}{2}\iint -\log \left|\frac{x-y}{|x-y|}+t\frac{\psi(x)-\psi(y)}{|x-y|}\right|~d\muv(x)d\muv(y)
\end{multline}
and then Taylor expand the logarithm to obtain 
\begin{equation*}
-t\iint \frac{x-y}{|x-y|^2} \cdot (\psi(x)-\psi(y))~d\muv(x)d\muv(y)+O\(t^2\|\psi\|_{C^{0,1}}^2\).
\end{equation*}
Next,
\begin{multline}
\int(V(x+t\psi(x))-V(x))~d\muv(x)+t\int \xi(x+t\psi(x))~d\muv(x)-t\int \xi(x)~d\muv(x)
=\\
t\int \nabla V(x)\cdot \psi(x)~d\muv(x)+O\(t^2\|\psi\|_{L^\infty}^2+t^2\|\xi\|_{C^{0,1}}\|\psi\|_{L^\infty}\).
\end{multline}
Since $\nabla \(\int -\log|x-y|~d\muv(y)+V(x)\)=0$ on $\supp(\muv)$ we have via differentiation and symmetry that
\begin{equation*}
-t\iint \frac{x-y}{|x-y|^2} \cdot (\psi(x)-\psi(y))~d\muv(x)d\muv(y)+t\int \nabla V(x)\cdot \psi(x)~d\muv(x)=0.
\end{equation*}
Finally, $\|\log \det D\phi_t\|_{L^\infty}\lesssim |t|\|\psi\|_{C^{0,1}}$. Combining all of these with the definition of $T_0$ in (\ref{t0}) yields the result.
\end{proof}

Next, we will control $T_1$ by choosing a transport $\psi$ that causes $T_1$ to vanish at order $t$. We will accomplish this by inverting (\ref{IE-MT}). First, we need the following careful estimate since the second derivatives of $\g$ are not bounded.
\begin{lem}[Careful Taylor Expansion]\label{CTE}
Let $|t|=o_N(1)$ and suppose that $\psi$ is Lipschitz. For any $|y-x|\geq \epsilon$ and large enough $N$,
\begin{align*}
\bigg|\g((x-y)+t(\psi(x)-\psi(y)))&-\g(x-y)-t\nabla \g(x-y)\cdot (\psi(x)-\psi(y))\bigg| \\
&\leq Ct^2\|\psi\|_{C^{0,1}}^2
\end{align*}
with constant independent of $\epsilon$ and $\beta$.
\end{lem}
\begin{proof}
With $|y-x|\geq \epsilon$ we have enough smoothness of $\g$ to Taylor expand; the quantity on the left hand side is the function minus its first Taylor polynomial, whose remainder is given by the integral remainder
\begin{equation*}
\sum_{i,j}v_iv_j\int_0^1 (1-s)(\partial_{i,j}\g)(\vec a+s\vec v)~ds,
\end{equation*}
where we have introduced the shorthand $\vec a=x-y$ and $\vec v=t(\psi(x)-\psi(y))$. Computing directly we have
\begin{equation*}
\partial_{i,j}\g(x)=\frac{-1}{|x|^{2}}\indic_{i=j}+\frac{2x_ix_j}{|x|^{4}}
\end{equation*}
and so the absolute value of the remainder is given by 
\begin{align*}
&\left|-\int_0^1(1-s)\frac{|\vec v|^2}{|\vec a+s\vec v|^2}~ds +\sum_{i,j}\int_0^1 \frac{2(1-s)(\vec a+s\vec v)_i(\vec a+s\vec v)_jv_iv_j}{|\vec a+s\vec v|^4}~ds\right|\\
&\lesssim \int_0^1(1-s)\frac{|\vec v|^2}{|\vec a+s\vec v|^2}~ds+\int_0^1 2(1-s) \frac{(\vec v \cdot (\vec a+s\vec v))^2}{|\vec a+s\vec v|^4}~ds \\
&\lesssim \int_0^1(1-s)\frac{|\vec v|^2}{|\vec a+s\vec v|^2}~ds.
\end{align*}
with constant independent of $\epsilon$. Next, since $\|t\psi\|_{C^{0,1}}=o_N( 1)$, we have for large enough $N$ that $|\vec a+c \vec v|\geq \frac{1}{2}|\vec a|$. Thus,
\begin{multline*}
\left|\g((x-y)+t(\psi(x)-\psi(y)))-\g(x-y)-t\nabla \g(x-y)\cdot (\psi(x)-\psi(y))\right| \\
\leq Ct^2\frac{|\psi(x)-\psi(y)|^2}{|x-y|^2} \leq Ct^2\|\psi\|_{C^{0,1}}^2
\end{multline*}
with constant independent of $\epsilon$, as desired.
\end{proof}

\begin{prop}[Control of $T_1$]\label{Control of t1}
Suppose $\psi$ is Lipschitz and let $T_1$ be as in (\ref{t1}). Suppose $\psi$ solves (\ref{IE-MT}).
Then,
\begin{equation}\label{t1 control}
|T_1|\lesssim \beta N^2t^2\(1+\| \xi\|_{C^{0,1}}\)\(\|\psi\|_{L^\infty}+\|\psi\|_{C^{0,1}}^2\)+|t|N\|\psi\|_{C^{0,1}}.
\end{equation}
In particular, if $|t|\sim \frac{1}{N}$, $T_1$ is order $1$.
\end{prop}
\begin{proof}
Let's start with 
\begin{equation*}
 -\beta N\int \left( \int (\g(\phi_t(x)-\phi_t(y))-\g(x-y))~d\muv(y)+(V_t \circ \phi_t -V)(x)\right)~d\fluct_N(x).
 \end{equation*}
 We can rewrite the integrand via Taylor expansion using Lemma \ref{CTE} as
 \begin{multline*}
\int \g(\phi_t(x)-\phi_t(y))~d\muv(y)+V(\phi_t(x))+t\xi(\phi_t(x)) -\int \g(x-y)~d\muv(y)-V(x) \\
=t\left[\int \nabla \g(x-y)\cdot (\psi(x)-\psi(y))~d\muv(y)+\nabla V(x)\cdot \psi(x)+\xi(x)\right] \\
+O\(t^2(1+\| \xi\|_{C^{0,1}})\(\|\psi\|_{L^\infty}+\|\psi\|_{C^{0,1}}^2\)\)
\end{multline*}
where we have taken $\epsilon \downarrow 0$ in the error terms in the expansion of $\g$, since the $O$ is independent of $\epsilon$. Note that the expression in brackets is equivalent to (\ref{IE-MT}), and in particular is constant by assumption (so its fluctuations are zero). Bounding the fluctuation measure by $\sim N$ we find 
\begin{multline}
\left| -\beta N\int \left( \int (\g(\phi_t(x)-\phi_t(y))-\g(x-y))~d\muv(y)+(V_t \circ \phi_t -V)(x)\right)~d\fluct_N(x)\right| \\
\lesssim \beta N^2t^2\(1+\| \xi\|_{C^{0,1}}\)\(\|\psi\|_{L^\infty}+\|\psi\|_{C^{0,1}}^2\).
\end{multline}
Finally, we use again that $\|\log \det D\phi_t\|_{L^\infty}\lesssim t\|\psi\|_{C^{0,1}}$ and control the fluctuation measure again by $\sim N$ to bound 
\begin{equation*}
\left|\int \log \det D\phi_t~d\fluct_N(x)\right|\lesssim |t|N\psi\|_{C^{0,1}},
\end{equation*}
as desired.
\end{proof}

We conclude by a careful analysis of $T_2$. The argument is again a Taylor expansion, appealing also now to the commutator type estimates of \cite{RS22}.
\begin{prop}[Control of $T_2$]\label{Control of t2}
Let $\psi$ be Lipschitz, and suppose $T_2$ is as in (\ref{t2}). Then,
\begin{equation*}
\|T_2\|\lesssim \beta N^2t^2\|\psi\|_{C^{0,1}}^2+\beta |t|N\|\psi\|_{C^{0,1}}+\beta |t|\|\psi\|_{C^{0,1}}\(\FN(\XN,\muv)+\frac{N}{2}\log N\right).
\end{equation*}
\end{prop}
\begin{proof}
We Taylor expand as before, writing 
\begin{align*}
\g(\phi_t(x)-\phi_t(y))-g(x-y)&=-t\frac{x-y}{|x-y|^2}\cdot (\psi(x)-\psi(y))+O\(t^2\|\psi\|_{C^{0,1}}^2\)\\
&=t \nabla \g(x-y)\cdot (\psi(x)-\psi(y))+O\(t^2\|\psi\|_{C^{0,1}}^2\).
\end{align*}
The error term is immediately controlled by $\beta N^2t^2\|\psi\|_{C^{0,1}}^2$ by controlling the fluctuation measure by $\sim N$. For the main term, we use the commutator estimate of \cite[Theorem 1.1]{RS22}:
\begin{equation*}
\left|\int \nabla \g(x-y)\cdot (\psi(x)-\psi(y))~d\fluct_N(x)d\fluct_N(y)\right|\lesssim \|\psi\|_{C^{0,1}}\left(\FN(\XN,\muv)+\frac{N}{2}\log N+N\right),
\end{equation*}
which establishes the result.
\end{proof}
Coupling all of this together yields Proposition \ref{genfluct}.
\begin{proof}[Proof of Proposition \ref{genfluct}]
Coupling Propositions \ref{extfluct}, \ref{expansion}, \ref{Control of t0}, \ref{Control of t1} and \ref{Control of t2} we find 
\begin{multline}\label{con}
\Esp \left[\exp \left(-\beta tN\Fluct_N(\xi)\right)\right]\lesssim\exp\biggl( \beta N^2t^2\(\|\psi\|_{L^\infty}+\|\psi\|_{C^{0,1}}^2+\|\xi\|_{C^{0,1}}\|\psi\|_{L^\infty}+\|\xi\|_{C^{0,1}}\|\psi\|_{C^{0,1}}^2\) \\
+\max(1,\beta)N|t|\|\psi\|_{C^{0,1}}\biggr)\Esp \left[\exp \left(\beta |t|\|\psi\|_{C^{0,1}}\(\FN(\XN,\muv)+\frac{N}{2}\log N\)\right)\right].
\end{multline}
Using the expansion of the partition function from \cite{SS15-2} or proceeding as in \cite[Lemma 2.15]{LS18}, we find
\begin{equation*}
\Esp \exp \left(t\left(\FN+\frac{N}{2}\log N\right)\right)\leq e^{C\frac{|t|}{\beta}N}
\end{equation*}
for any $|t|\leq \frac{\beta}{2}$. Inserting this into (\ref{con}) and simplifying somewhat, we find 
\begin{multline*}
\Esp \left[\exp \left(-\beta tN\Fluct_N(\xi)\right)\right]\lesssim \\
\exp\( \beta N^2t^2\(\( \|\psi\|_{L^\infty}+\|\psi\|_{C^{0,1}}^2\)(1+\|\xi\|_{C^{0,1}})
\)+\max(1,\beta)N|t|\|\psi\|_{C^{0,1}}\),
\end{multline*}
as desired.
\end{proof}
Constructing a transport as in \cite[Lemma 3.4]{LS18} allows us to invert (\ref{IE-MT}) and obtain a concentration result with less regularity on the test function $\xi$. This is Proposition \ref{Coro1}.
\begin{proof}[Proof of Proposition \ref{Coro1}]
We can apply Proposition \ref{genfluct} as soon as we can invert (\ref{IE-MT}). This can be accomplished using the transport map introduced in \cite[Theorem 3.4]{LS18}, which solves 
\begin{equation}\label{PDE}
\begin{cases}
\div(\muv \psi)=-\frac{1}{\cd}\Delta \xi & \text{ in }\Sigma_i \\
\psi \cdot \vec n=\frac{1}{\cd \muv}[\nabla \xi^\Sigma]\cdot \vec n& \text{ on }\partial \Sigma_i \\
\psi=\(\xi^\Sigma-\xi\)\frac{\nabla \zeta_V}{|\nabla \zeta_V|^2}+\psi^\perp & \text{in }\Sigma^c
\end{cases}
\end{equation}
on each $\Sigma_i$, where $\psi^\perp$ is orthogonal to $\nabla \zeta_V$ and is chosen to make the transport map continuous at the boundary in the tangential direction (it is continuous in the normal direction automatically by the behavior of $\nabla \zeta_V$ at the interface). $[\nabla \xi^\Sigma]$ denotes the jump in the harmonic extension across the interface of the droplet. 

In the one-cut regime, no additional assumptions are needed to guarantee a solution since $\int_{\partial \Sigma}\nabla \xi^\Sigma \cdot \hat{n}$=0; this can be seen by integrating by parts in $B_R$ with $R \rightarrow +\infty$ and using classical gradient estimates on harmonic functions. In the multicut regime this may not be true on every connected component, hence the additional assumption (\ref{MC}) as in \cite[Lemma 3.4]{LS18}.

Let us discuss briefly why the transport map above inverts our equation (\ref{IE-MT}). Recall that we want to solve \begin{equation*}
\psi \cdot \nabla \zeta_V+\xi-\int \nabla \g(x-y)\cdot\psi\muv(y)=c_\xi
\end{equation*}
for some constant $c_\xi$. In $\Sigma$, $\nabla \zeta_V=0$; integrating by parts in $y$ and substituting (\ref{PDE}) inside $\Sigma$ and on $\partial \Sigma$ we find
\begin{align*}
\xi(x)-\int \nabla \g(x-y)\cdot \psi\muv(y)&=\xi(x)-\int_\Sigma \g(x-y)\div(\psi \muv)(y)+\int_{\partial \Sigma}\g(x-y)\psi \muv(y)\cdot \hat{n} \\
&=\xi(x)-\int_{\Sigma}\g(x-y)\frac{-\Delta \xi}{\cd}(y)+\int_{\partial \Sigma}\g(x-y)[\nabla \xi^\Sigma] \cdot \hat n.
\end{align*}
Outside, $\psi \cdot \nabla \zeta_V=\xi^\Sigma-\xi$, so 
\begin{equation*}
\psi \cdot \nabla \zeta_V+\xi-\int \nabla \g(x-y)\cdot\psi\muv(y)=\xi^\Sigma(x)-\int_{\Sigma}\g(x-y)\frac{-\Delta \xi}{\cd}(y)+\int_{\partial \Sigma}\g(x-y)[\nabla \xi^\Sigma] \cdot \hat n.
\end{equation*}
Let $w(x)$ denote 
\begin{equation}\label{w}
w(x)=\begin{cases}
\xi(x)-\int_{\Sigma}\g(x-y)\frac{-\Delta \xi}{\cd}(y)+\int_{\partial \Sigma}\g(x-y)[\nabla \xi^\Sigma] \cdot \hat n & \text{if }x\in \Sigma \\
\xi^\Sigma(x)-\int_{\Sigma}\g(x-y)\frac{-\Delta \xi}{\cd}(y)+\int_{\partial \Sigma}\g(x-y)[\nabla \xi^\Sigma] \cdot \hat n & \text{if }x \in \Sigma^c.
\end{cases}
\end{equation}
By classical results on single-layer potentials (cf. \cite[Sec. II.3]{DL90}, \cite[Chapter 3]{F95} and \cite[Appendix A]{SS18} for a review) $\int_{\partial \Sigma}\g(x-y)[\nabla \xi^\Sigma] \cdot \hat n$ is harmonic in $\Sigma$ and $\Sigma^c$. Furthermore, for $x \in \Sigma$
\begin{equation*}
\Delta \left(\xi(x)-\int_{\Sigma}\g(x-y)\frac{-\Delta \xi}{\cd}(y)\right)=\Delta \xi-\Delta \xi=0
\end{equation*}
and for $x \in \Sigma^c$, 
\begin{equation*}
\Delta \left(\xi^\Sigma(x)-\int_{\Sigma}\g(x-y)\frac{-\Delta \xi}{\cd}(y)\right)=0-0=0
\end{equation*}
so $w$ is harmonic in $\Sigma$ and $\Sigma^c$. Since $\nabla \g$ is locally integrable, $\int_{\Sigma}\g(x-y)\frac{-\Delta \xi}{\cd}(y)$ has a continuous normal derivative across the interface $\partial \Sigma$. Using classical formulae for the jump in the normal component of the gradient of a single layer potential (cf. \cite[Theorem A.1]{SS18}, \cite{DL90})] we have
\begin{equation*}
\partial_{\hat{n}, \text{out}}w-\partial_{\hat{n}, \text{in}}w=[\nabla \xi^\Sigma]\cdot \hat{n}-[\nabla \xi^\Sigma]\cdot \hat{n}=0
\end{equation*}
where $\partial_{\hat{n}, \text{out}}w$ and $\partial_{\hat{n}, \text{in}}w$ denote the normal derivatives of $w$ taken from outside and inside $\Sigma$, respectively. So, no divergence is created at the boundary when we make the piecewise definition (\ref{w}) for $w(x)$. As a result, $w(x)$ is harmonic on $\R^2$ and is bounded by $O(\log|x|)$ as $|x|\rightarrow+\infty$ and is therefore constant by Liouville's theorem. Hence, the map $\psi$ given in (\ref{PDE}) solves (\ref{IE-MT}).

Finally, since $\Delta \xi \in C^{0,\alpha}(\Sigma)$ we have the Schauder estimate $\|\psi\|_{C^{1,\alpha}(\Sigma)}\lesssim \|\Delta \xi\|_{C^{0,\alpha}(\Sigma)}$. In $U \setminus \Sigma$ the map is Lipschitz away from the boundary; we also know that it is continuous across to the boundary by choice of the boundary condition and so we have the estimate 
\begin{equation*}
\|\psi\|_{C^{0,1}(U)}\lesssim \|\xi\|_{C^{0,1}(U)}+\|\Delta \xi\|_{C^{0,\alpha}(\Sigma)}
\end{equation*}
Substituting this into (\ref{concentration}) yields the estimate (\ref{corcon}).
\end{proof}

\bibliographystyle{amsalpha}
\bibliography{GMC}{}

\end{document}